\documentclass{article}
\usepackage[english]{babel}
\usepackage[letterpaper,top=2cm,bottom=2cm,left=3cm,right=3cm,marginparwidth=1.75cm]{geometry}
\usepackage{amsmath,amssymb, amsfonts,textcomp, amsthm, mathtools}
\usepackage{graphicx}
\usepackage{comment}
\usepackage{float}
\usepackage{tablefootnote}
\usepackage{longtable}
\usepackage{enumerate}

\usepackage{boldline} 
\usepackage{setspace}
\usepackage[hypertexnames=false]{hyperref}

\theoremstyle{plain}
\newtheorem{theorem}{Theorem} 
\newtheorem{lemma}[theorem]{Lemma}
\newtheorem{corollary}[theorem]{Corollary}

\newtheorem{question}[theorem]{Question}

\theoremstyle{definition}
\newtheorem{definition}[theorem]{Definition}
\newtheorem{example}[theorem]{Example}

\theoremstyle{remark}

\DeclareMathOperator{\Span}{span}
\author{K.S. Enoch Lee}
\title{A note on generalized distributive sets of a finite Dickson nearfield}

\begin{document}
\emergencystretch 3em
\maketitle
\begin{abstract}Let $\mathbb{F}_{q^n}$ be  a finite field  where $q=p^l$ ($p$ is a prime)  and a Dickson nearfield $R$ associated to the field. 
In this note we investigate the basic properties of the  generalized distributive set $D(a,b)$ where $a,b\in R$. The set is well-known when $n=1, 2$. For $n>2$, we characterize when $D(a,b)$ is a subfield of $\mathbb{F}_{q^n}$. Furthermore we find the unique largest subfield of $\mathbb{F}_{q^n}$ contained in $D(a,b)$. We also provide a condition to guarantee that $D(a,b)$ is not a subfield of  $\mathbb{F}_{q^n}$. We end the note by posing the question of  when $D(a,b)$, as a space over $\mathbb{F}_p$, has a primitive normal basis.
  \end{abstract}
In the following we agree that $\mathbb{F}_m=(\mathbb{F}_{m},+,\cdot)$ is the finite field of order $m$ and $(q,n)$ is a Dickson pair~\cite{DJAGBA20} or a pair of Dickson numbers~\cite{Pilz83} where $q=p^l$ for some prime $p$.
Let $DN(q,n)$ be the collection of all Dickson nearfields~\cite{DJAGBA20} associated to the Dickson pair $(q,n)$.
Denote by 
$g\in \mathbb{F}_{q^n}$ a primitive element (i.e., a generator of the non-zero elements of the field) with the primitive polynomial $p(x)$. In this note a nearring is a zero-symmetric 
left nearring unless stated otherwise. Thus we only consider nearfields that are zero-symmetric. 
We denote by
 $DN_g(q,n)=(DN_g(q,n),+,\circ)$ the Dickson nearfield~\cite{DJAGBA20}, \cite{Pilz83} associated to $g$. The underlying set of $DN_g(q,n)$ and that of $\mathbb{F}_{q^n}$ are the same. 
Therefore we might use $DN_g(q,n)$ or $\mathbb{F}_{q^n}$ to denote the underlying set should there be no confusion. 
Furthermore $$DN(q,n)=\{DN_g(q,n) \mid g \text{ is a primitive element of }\mathbb{F}_{q^n}\}.$$
We write $ab$ or  $a\cdot b$ for the field multiplication while $a\circ b$ the Dickson nearfield multiplication. Here,  $a^{-1}$ denotes the
multiplicative inverse of $a$ with respect to the field multiplication while ${a^{\circ(-1)}}$ is the   Dickson nearfield multiplicative inverse instead.

For any non-zero $a,b\in DN_g(q,n)$, the generalized distributive set associated to the pair $a$ and $b$ is defined by
 $$D(a,b)=\{d\in DN_g(q,n) \mid (a+b)\circ d =a\circ d +b\circ d \}.$$ 
Immediately we have $(D(a,b),+)$ is an abelian group.   If there is no confusion, when we write $D(a,b)$ or $DN_g(q,n)$ we assume a Dickson pair  $(q,n)$ where $q=p^l$ for some prime $p$.

For the purpose of discussion we stipulate that $H=\langle g^n \rangle$ the multiplicative subgroup of $(\mathbb{F}^*_{q^n},\cdot)$ generated by $g^n$ where $\mathbb{F}^*_{q^n}$ means the non-zero element of $\mathbb{F}_{q^n}$.
Below we list some known results.

\begin{lemma}\cite[Lemma 3.5]{DJAGBA20} Suppose $a$ and $b\in DN_g(q,n)$. If $a$, $b$, and $a+b$  all belong to the same $H$-coset, then
$(a+b)\circ d=a\circ d+b\circ d$ for any $d\in DN_g(q,n)$.  Hence $(D(a,b), +, \cdot)=(\mathbb{F}_{q^n},+,\cdot)$. 
	\end{lemma}

        For the special case when $n=1$, the pair $(q,1)$ is always a Dickson pair.  Moreover 
$$ x\circ y= x\cdot y.$$ Therefore the Dickson nearfield is the field $\mathbb{F}_q$ and we have $D(a,b)=\mathbb{F}_q$.

For the case when  $n=2$, we have   
$$x\circ d =\begin{cases} 0 & \text{ if }x=0  \\ x\cdot d & \text{ if }x\in H\\ x\cdot d^q& \text{ if } x\notin H \end{cases} $$
Let $\phi(x) \in \{1, q\}$ denote the Frobenius exponent associated with $x \in \mathbb{F}_{q^2}$, so that $x \circ d = x d^{\phi(x)}$. The defining relation $(a+b)\circ d = a\circ d + b\circ d$ expands to:$$(a+b) d^{\phi(a+b)} = a d^{\phi(a)} + b d^{\phi(b)}$$
If $a, b,$ and $a+b$ all belong to the same coset (or are $0$), then $\phi(a)=\phi(b)=\phi(a+b)=\sigma \in \{1, q\}$. Substituting this gives:$$(a+b)d^\sigma = ad^\sigma + bd^\sigma = (a+b)d^\sigma$$
This identity holds trivially for all $d \in \mathbb{F}_{q^2}$, so $D(a,b) = \mathbb{F}_{q^2}$. Note if $b = -a$, then $a+b=0$. Since $-1 \in H$ for all odd $q$, $-a$ belongs to the same coset as $a$, making $D(a,-a) = \mathbb{F}_{q^2}$ as well.

If $a, b,$ and $a+b$ are not all in the same coset, exactly two of the elements share an exponent while the third has the other exponent from $\{1, q\}$.
\begin{itemize}
\item If $\phi(a) = \phi(b) \neq \phi(a+b)$:
  $$(a+b) d^{\phi(a+b)} = (a+b) d^{\phi(a)}$$
  Since $a+b \neq 0$, dividing by $(a+b)$ yields $d^{\phi(a+b)} = d^{\phi(a)}$.
\item If $\phi(a) = \phi(a+b) \neq \phi(b)$:
  $$a d^{\phi(a)} + b d^{\phi(a)} = a d^{\phi(a)} + b d^{\phi(b)} \implies b d^{\phi(a)} = b d^{\phi(b)}$$
  Since $b \neq 0$, dividing by $b$ yields $d^{\phi(a)} = d^{\phi(b)}$.
\end{itemize}
Thus the relation reduces directly to $d^q = d$. The set of solutions to $d^q = d$ in $\mathbb{F}_{q^2}$ is precisely the subfield $\mathbb{F}_q$, so the set $D(a,b) = \mathbb{F}_q$.

\begin{theorem}\cite[Theorem 3.13]{DJAGBA20} Assume $n>2$. Suppose $a$ and $b\in DN_g(q,n)$. If at least two of $a$, $b$,
and $a+b$ belong to the same $H$-coset, then  $(D(a,b), +, \cdot)$ is a subfield of $(\mathbb{F}_{q^n},+,\cdot)$. 
\end{theorem}

Hereafter we assume $n>2$. 
 A routine checking shows the following result.

\begin{lemma}
$D(1,b)$ is closed under the nearfield multiplication $\circ$ if and only if $D(1,b)\subseteq D(d , b\circ d)=D(1,d^{\circ ^{-1}}\circ b\circ d)$ for any $d\in D(1,b)$.
\end{lemma}

\begin{lemma}
Assume $0\neq a\in DN_g(q,n)$. Then $D(1,b)=D(a,a\circ b)$ for any $b\in DN_g(q,n)$. 
As a consequence $D(a,b)=D(1,a^{\circ(-1)}\circ b)$ whenever $a\neq 0$.
\begin{proof}
Since $a\neq 0$, the multiplicative inverse with respect to $\circ$ exists. It is quite easy to see that the stated result is a consequence of the following equivalence: 
\begin{enumerate}[(1)] 
\item  $d\in D(1,b)$.
\item $(1+b)\circ d=d+b\circ d$.
\item $a\circ ((1+b)\circ d)=a\circ (d+b\circ d)$.
\item $(a+a \circ b)\circ d=a\circ d+a\circ b\circ d$.
\item  $d\in D(a, a\circ b)$.
\end{enumerate}
\end{proof}
\end{lemma}

Hence we shall investigate the special case $D(1,b)$ instead of the general case. We list the following well-known number theoretic  results without proof.
 
\begin{lemma}Suppose $a$, $m$, and $n$ are positive integers. Then 
\begin{enumerate}[(1)]
\item $\gcd(a^m-1, a^n-1)=a^{\gcd(m,n)}-1$. 
\item $a^m-1 \mid a^n-1$ if and only if $m \mid n$.
\end{enumerate}
\end{lemma}

Hereafter we assume $n>2$. Suppose  $b$ and  $1+b$ are non-zero elements of $DN_g(q,n)$. We stipulate that  $s,t$ are non-negative integers that  $b\in g^{\frac{q^s-1}{q-1}}H$ and $1+b\in g^{\frac{q^t-1}{q-1}}H$.

\begin{lemma}\label{Dabsq}Let $b$ and $1+b$ be non-zero elements of $DN_g(q,n)$.
Suppose $0\neq u\in D(1,b)$  and $\mu= \gcd(s,t,n)$, then the following are equivalent:
\begin{enumerate}[(1)]
\item $u^2\in D(1,b)$;
\item $u^{q^\mu-1}=1$ (note $u^{q^\mu-1}=1\Leftrightarrow  u^{q^s-1}=1\Leftrightarrow u^{q^t-1}=1$);
\item $u^{-1}\in D(1,b)$.
\end{enumerate}
\begin{proof}Note $(1+b)\circ u -u - b\circ u=(1+b)u^{q^t}-u - bu^{q^s}=0$. This implies
    \begin{equation} u^{q^t}-u=b(u^{q^s}-u^{q^t})
   \label{eq:1}\tag{Eq:1}
         \end{equation}

$(1) \implies (2)$ From $(1)$, we have $(1+b)\circ u^2 -u^2 - b\circ u^2=(1+b)u^{2q^t}-u^2 - bu^{2q^s}=0$. Thus
\begin{equation}     (u^{q^t}-u)(u^{q^t}+u)=b(u^{q^s}-u^{q^t})(u^{q^s}+u^{q^t})
     \label{eq:2}\tag{Eq:2}
     \end{equation}
     If $u^{q^t}\neq u$, then \eqref{eq:1} implies $u^{q^t}\neq u^{q^s}$. From \eqref{eq:2} we then have
     $u^{q^t}+u=u^{q^s}+u^{q^t}$. In other words, $u^{q^s}=u$. From \eqref{eq:1} we now have $b=-1$. This contradicts the fact that $1+b\neq 0$. We conclude that $u^{q^t}=u$. This means $1=u^{q^t-1}=u^{q^s-1}$. Since $u^{q^n-1}=1$, we have $1=u^{\gcd(q^s-1,q^t-1,q^n-1)}=u^{q^{\gcd(s,t,n)}-1}$ by the previous lemma.

$(2) \implies (3)$
    Note $(u^{-1})^{q^\mu} = (u^{q^\mu})^{-1}=u^{-1}$. Thus $(u^{-1})^{q^t}=(u^{-1})^{q^s}=u^{-1}$. We then have
    $(1+b)\circ u^{-1} -u^{-1} - b\circ u^{-1}=(1+b)(u^{-1})^{q^t}-u^{-1} - b(u^{-1})^{q^s}=(1+b)u^{-1}-u^{-1} - bu^{-1}=0$.

$(3) \implies (1)$
Since $u^{-1}\in D(1,b)$, we have $(1+b)u^{-q^t}-u^{-1} - bu^{-q^s}=0$. Therefore 
$0=(1+b)u^{q^s}-u^{q^s+q^t-1}-bu^{q^t}$. Combine with \eqref{eq:1} we have
$0=u^{q^s+q^t-1}+u-u^{q^t}-u^{q^s}=u(u^{q^s-1}-1)(u^{q^t-1}-1)$. This means 
$u^{q^s-1}=1$ or $u^{q^t-1}=1$. Again from  \eqref{eq:1} we have $u^{q^s-1}=u^{q^t-1}=1$. Now 
$$(1+b)\circ u^2 -u^2 - b\circ u^2=(1+b)u^{2q^t}-u^2 - bu^{2q^s}=(1+b)u^{2}-u^2 - bu^{2}=0.$$ Thus we have $(1)$.

The equivalences stated in $(2)$ follow easily from \eqref{eq:1}.
\end{proof}
\end{lemma}

\begin{theorem}\label{vectorSubspacequ} 
  Suppose $b$ and $1+b$ are non-zero elements of $DN_g(q,n)$.
Let $\mu= \gcd(s,t,n)$. Then $\mathbb{F}_{p^m} \subseteq D(1,b)$ when $m\mid l\mu$. Furthermore $(D(1,b),+)$ is a 
subspace  of $\mathbb{F}_{q^n}$ over the subfield $\mathbb{F}_{p^m}$ (here the scalar multiplication is the field multiplication).  Moreover if
$D(1,b)$ is closed under the field multiplication, then $(\mathbb{F}_{q^\mu}, +,\cdot) = (D(1,b), +, \cdot)$. 
\begin{proof} Clearly we have $\mathbb{F}_{p^m}$ a subfield of $\mathbb{F}_{q^n}$.
Let $0\neq d\in \mathbb{F}_{p^m}$. So $d^{q^t}=d^{q^s}=d$. Hence $(1+b)\circ d-d-b\circ d=
(1+b)d^{q^t}-d-bd^{q^s}=(1+b)d-d-bd=0$. We have $\mathbb{F}_{p^m} \subseteq D(1,b)$. Now it suffices to check if $dw\in D(1,b)$ for any $d\in \mathbb{F}_{p^m}$ and $w\in D(1,b)$.  Clearly
$(1+b)(dw)^{q^t}-(dw)-b(dw)^{q^s}=d[(1+b)w^{q^t}-w-bw^{q^s}]=d[(1+b)\circ w-w-b\circ w]=0$. Thus $(D(1,b),+)$ is a 
subspace  of $\mathbb{F}_{q^n}$ over the subfield $\mathbb{F}_{p^m}$.

Suppose $D(1,b)$ is closed under the field multiplication. We have  just shown that $\mathbb{F}_{q^\mu}\subseteq D(1,b)$.
It remains to show $D(1,b)\subseteq \mathbb{F}_{q^\mu}$. Let $u\in D(1,b)$, so $u^2\in D(1,b)$. Lemma~\ref{Dabsq} implies $u^{q^{\gcd(s,t,n)}-1}=1$. In other words, $u\in \mathbb{F}_{q^\mu}$. We are done!
\end{proof} 
\end{theorem}

 For purpose of investigation we view $D(1,b)$ as a vector space over the prime subfield $\mathbb{F}_p$.  Let ${\mathbf{B}}$ be a basis of the space  $D(1,b)$ over the prime subfield $\mathbb{F}_p$. Suppose $b$ and $1+b$ are non-zero. If $\lvert {\mathbf{B}}\rvert =l\mu$, then Theorem~\ref{vectorSubspacequ} implies that  $D(1,b)$ is closed under the field multiplication, and is indeed the field $\mathbb{F}_{q^\mu}$.
\begin{lemma}\label{D1bBase}
Let $\mathbf{B}$ be a basis of the vector space $D(1,b)$ over $\mathbb{F}_p$. We have $D(1,b)$ is not closed under the field multiplication if and only if 
there are $u_i,u_j\in \mathbf{B}$ such that $u_iu_j\not \in D(1,b)$.
\begin{proof}
Let $\mathbf{B}=\{u_1,u_2,\dots, u_m\}$. Obviously $\mathbf{B}\subseteq D(1,b)$. 

Assume $D(1,b)$ is not closed under the field multiplication. Let $\alpha_i, \beta_i\in \mathbb{F}_p$ such that 
$$(\sum \alpha_iu_i)(\sum \beta_iu_i)\not \in D(1,b).$$  In view of Theorem~\ref{vectorSubspacequ}, we know 
$\sum \alpha_iu_i$ and $\sum \beta_iu_i\in D(1,b)$. Note 
$$
(\sum \alpha_iu_i)(\sum \beta_iu_i)=
 (\sum \alpha_i\beta_iu_i^2) +(\sum\limits_{i\neq j} (\alpha_i\beta_j+\alpha_j\beta_i)u_iu_j)  $$
We either have $0\neq \alpha_i\beta_iu_i^2\not \in D(1,b)$ or  $0\neq (\alpha_i\beta_j+\alpha_j\beta_i)u_iu_j\not \in D(1,b)$ where $i\neq j$. 
In any case we have $u_iu_j\not \in D(1,b)$ for some $i$ and $j$. The converse is obvious.

\end{proof}
\end{lemma}

\begin{theorem}\label{Dabsqclosed}
  If $u$ and $v\in D(1,b)$ are nonzero such that $u^2,v^2\in D(1,b)$, then $uv\in D(1,b)$.
\begin{proof}There is nothing to prove if $b$ or $1+b=0$. Thus we assume $b$ and $1+b$ are non-zero. 
 Lemma~\ref{Dabsq} implies $1=u^{q^t-1}=u^{q^s-1}=v^{q^t-1}=v^{q^s-1}$.
 Using this fact one can easily show that
     \begin{align*} 
(1+b)\circ (uv) -uv -b\circ (uv) &=  (1+b)(uv)^{q^t}-uv-b(uv)^{q^s} \\ 
 &=   (1+b)(uv)-uv-b(uv) \\
 &=0
\end{align*}
    Therefore $uv\in D(1,b)$.

\end{proof}
\end{theorem}

\begin{corollary}
  Suppose $\mathbf{B}$ is a basis of the space $D(1,b)$ over the prime subfield $\mathbb{F}_p$. Then $D(1,b)$ is closed  under the field multiplication (i.e. $(D(1,b),+,\cdot)$ is a field) if and only if $u^2\in D(1,b)$ for any $u\in \mathbf{B}$ if and only if $\lvert \mathbf{B}\rvert =l\mu$ where $\mu= \gcd(s,t,n)$.
\end{corollary}

\begin{lemma}
When $n\mid q-1$ we have  $\frac{q^k-1}{q-1}\equiv k \pmod{n}$ for any non-negative integer $k$.
\begin{proof}
$\frac{q^k-1}{q-1}=1+q+q^2+\cdots+q^{k-1}\equiv k \pmod{n}$ since $q\equiv 1\pmod{n}$.

\end{proof}
\end{lemma}

The following example shows that even if $q-1=n$ we cannot guarantee that $D(a,b)$ 
is closed under the field multiplication, thus cannot be a subfield of $\mathbb{F}_{q^n}$. 
\begin{example}
Let $q=7$ and $n=6$. Thus $(7,6)$ is a Dickson Pair.  Let $g\in \mathbb{F}_{q^n}$ is a 
primitive element and  is a root to the primitive polynomial 
$p(x)=x^6+x^4+5x^3+4x^2+6x+3$. Let $a=1$, $b=g^{75073}\in gH$. 

Tedious computations show
$b=g^3 + 2g^2 + 4g + 5$
 and  $1+b=g^{27850}=g^3 + 2g^2 + 4g + 6\in g^4H$.
 The above lemma implies 
 $r=0$, $s=1$ and $t=4$. 
For simplicity we can write an element of $\mathbb{F}_{q^n}$ in a vector form, for example, $b=\begin{bmatrix}0&0&1& 2& 4& 5\end{bmatrix}$
 where each entry is an element of the prime subfield.

Recall  the space $D(1,b)$ over the prime subfield $\mathbb{F}_7$ is 
the kernel of the linear map $\phi$ on the space $\mathbb{F}_{7^6}$ over its prime subfield $\mathbb{F}_7$ 
 defined by (see \cite[p.149]{DJAGBA20})
$$\lambda\mapsto (a+b)\circ \lambda - a\circ \lambda -b\circ \lambda=(a+b)\lambda^{7^3}-a -b\lambda^{7^1}.
$$ 
 The map $\phi $ can be represented by   
$$
M=\left[\begin{matrix}6 & 3 & 1 & 0 & 2 & 0\\1 & 1 & 1 & 0 & 5 & 0\\5 & 4 & 1 & 0 & 4 & 0\\1 & 0 & 0 & 1 & 4 & 0\\5 & 2 & 0 & 0 & 6 & 0\\2 & 6 & 0 & 1 & 1 & 0\end{matrix}\right]
$$
a $6\times 6$ matrix over $\mathbb{F}_7$. For example, 
$\phi(b)=Mb^T=\left[\begin{matrix}2 &0 & 3 & 4 & 3 & 6\end{matrix}\right]^T
=2g^5 + 3g^3 + 4g^2 + 3g + 6$. Since $D(1,b)$ is the kernel of $\phi$,  it is 
the nullspace of $M$. A quick calculation yields a basis of the space $D(1,b)$ over $\mathbb{F}_7$: 
$$\mathbf{B}=\{\begin{bmatrix}
1 & 1 & 5 & 6 & 0 & 0\end{bmatrix}, \begin{bmatrix}0 & 0 & 0 & 0 & 0 & 1\end{bmatrix}\}=\{g^5+g^4+5g^3+6g^2,1\}.$$ 
We have $\lvert \mathbf{B} \rvert=2 \neq l\gcd(s,t,n)=1$. Theorem~\ref{vectorSubspacequ} implies this $D(1,b)$ is not closed under the field multiplication. Furthermore $\lvert D(1,b)\rvert=7^2$.
Let
$u=g^5+g^4+5g^3+6g^2$. Then $u^2=  2g^5 + 5g^4 + 5g^3 + 3g^2 + 3g + 3 =\begin{bmatrix}2&5&5&3&3&3   \end{bmatrix}$. But we have
$$
M{(u^2)}^T=\begin{bmatrix}3&6&5&3&3&5 \end{bmatrix}^T\neq 0, \text{ i.e. } d^2\not \in D(1,b).
$$

\end{example}

\begin{lemma}Suppose $u,v\in D(1,b)$ such that $u^2\in D(1,b)$ but $v^2\not \in D(1,b)$. Then $(u+v)^2\not \in D(1,b)$.
\begin{proof}
Suppose $(u+v)^2 \in D(1,b)$. Theorem~\ref{Dabsqclosed} implies $u(u+v)\in D(1,b)$. Since $(D(1,b),+)$ is an abelian group, we have $v^2\in D(1,b)$.
\end{proof}
\end{lemma}

\begin{definition}
Suppose $\mathbf{B}$ is a basis of the space $D(1,b)$ over the prime subfield $\mathbb{F}_p$. The set $\mathbf{B}_1=\{u\in \mathbf{B} \mid u^2\in D(1,b)\} $ is called the {\em square closed set} of $\mathbf{B}$. 

We stipulate that $\Span(X)$ is the subspace of $D(1,b)$ over $\mathbb{F}_p$ spanned by elements of the set $X\subseteq D(1,b)$.
\end{definition}

Obviously $D(1,b)$ is closed under the field multiplication if and only if $D(1,b)=\Span(\mathbf{B}_1)$.

\begin{lemma}
Suppose $u\in \Span(\mathbf{B}_1)$. Then $u^2\in D(1,b)$.
\begin{proof}
Let $u=\sum\limits_{u_i\in \mathbf{B}_1}\alpha_i u_i$ for some $\alpha_i\in \mathbb{F}_p$. Theorem~\ref{Dabsqclosed} implies $u_iu_j\in D(1,b)$ whenever $u_i^2, u_j^2\in \mathbf{B}_1$. Thus $u^2\in D(1,b)$.
\end{proof}
\end{lemma}

\begin{theorem}\label{maxsqclosedset}
  Suppose $b$ and $1+b$ are non-zero elements of $DN_g(q,n)$. Let $\mathbf{B}$ be a basis of the space $D(1,b)$ over the prime subfield and $\mathbf{B}_1$ be the square closed set of $\mathbf{B}$. Then
$(\Span(\mathbf{B}_1),+,\cdot)$ is a subfield of $\mathbb{F}_{q^n}$ and $\lvert \mathbf{B}_1\rvert\leq l\mu$. Furthermore,  it is possible to construct a  basis $\mathbf{C}$ of $D(1,b)$ such that  $\lvert \mathbf{C}_1\rvert=l\mu$ where $\mathbf{C}_1$ is the square closed set of $\mathbf{C}$. In this case  $span(\mathbf{C}_1)=\mathbb{F}_{q^\mu}$ where $\mu=\gcd(s,t,n)$.
\begin{proof} Claim: $\lvert \mathbf{B}_1\rvert\leq l\mu$. The above lemma shows that $\Span(B_1)$ is closed under the field multiplication and thus $(span(B_1),+,\cdot)$ is a subfield of $\mathbb{F}_{q^n}$ in $D(1,b)$. Lemma~\ref{Dabsq} implies $\Span(\mathbf{B}_1)\subseteq \mathbb{F}_{p^{l\mu}}$. Thus we have the claim that  $\lvert \mathbf{B}_1\rvert  \leq  l\mu$. Theorem~\ref{vectorSubspacequ} implies $\mathbb{F}_{p^{l\mu}}\subseteq D(1,b)$. Let $\mathbf{C}$ be a basis of $D(1,b)$ extended from a basis of $\mathbb{F}_{p^{l\mu}}$  over the prime subfield.  Since the square of any element of $\mathbb{F}_{p^{l\mu}}$  belongs to $D(1,b)$,  we have that
 the square closed set $\mathbf{C}_1$ of $\mathbf{C}$ is of size $\geq l\mu$. Thus $\lvert \mathbf{C}_1\rvert =l\mu$.
\end{proof}
\end{theorem}

Obviously,  $D(1,b)$ is not a field if and only if $l\mu=\lvert\mathbf{B}_1\rvert<\lvert \mathbf{B}\rvert<ln$.

Suppose $\mathbf{B}$ and $\mathbf{B}'$ are bases of the space $D(1,b)$ over the prime subfield $\mathbb{F}_p$. Let $\mathbf{B}_1$ and $\mathbf{B}'_1$ be the square closed sets of $\mathbf{B}$ and $\mathbf{B}'$ respectively. Is $\lvert \mathbf{B}_1\rvert=\lvert \mathbf{B}'_1\rvert=l\mu$? The answer is no in general. 
\begin{example}
Consider the example  $(q,n)=(4,9)$ listed in the Table~\ref{tablenotfield}. The square closed set of $\mathbf{B}$ is $\mathbf{B}_1=\{1\}$ while $l\mu=2$. Therefore $\Span({\mathbf{B}_1})=\mathbb{F}_2\subseteq  D(1,b)$. The above theorem illustrates a simple process to construct a basis such that its square closed set spans the field $\mathbb{F}_{2^2}$. In other words $\mathbb{F}_{2^2}\subseteq D(1,b)$.  
 It should be noted that $\mathbb{F}_{2^3}$ is a subfield of $\mathbb{F}_{4^9}$ but not contained in $D(1,b)$ since $g^{\frac{2^{18}-1}{2^3-1}}=g^{37449} \in \mathbb{F}_{2^3}$ but  $g^{37449} \not \in D(1,b)$.  In fact, $\mathbf{B}\cap \mathbb{F}_{4^3}=\{1\}$, i.e. $D(1,b)\not \subseteq  \mathbb{F}_{4^3}$. Hence there is no proper subfield of $\mathbb{F}_{4^9}$  containing $D(1,b)$.
\end{example}

\begin{lemma}Suppose $b\in DN_g(q,n)$.
  The following are equivalent.
\begin{enumerate}[(1)]
\item For any bases $\mathbf{B}$ and $\mathbf{B}'$ of the space $D(1,b)$ over $\mathbb{F}_p$, we have $\Span(\mathbf{B}_1)=\Span(\mathbf{B}'_1)$ where $\mathbf{B}_1$ and $\mathbf{B}'_1$ are the square closed sets of $\mathbf{B}$ and $\mathbf{B}'$ respectively. 
\item For any basis $\mathbf{B}$ of the space $D(1,b)$ over $\mathbb{F}_p$, we have $\Span(\mathbf{B}_1)=\mathbb{F}_{p^{l\mu}}$ where $\mathbf{B}_1$ is the square closed set of $\mathbf{B}$. (Equivalently $\lvert \mathbf{B}_1\rvert=l\mu$.)
\end{enumerate}
\begin{proof} It suffices to consider the case when $b$ and $1+b$ are non-zero. Recall $\mu=\gcd(s,t,n)$. 

 Assume (1). Let $\mathbf{B}$ be a basis of $D(1,b)$ over $\mathbb{F}_p$.
 Theorem~\ref{maxsqclosedset} implies there is a basis $\mathbf{C}$ of $D(1,b)$ such that the size of its square closed set $\mathbf{C}_1$ is the largest possible, i.e. $l\mu$. Hence (2) is true. The converse is obvious.
Thus (1) and (2) are equivalent.

\end{proof}
\end{lemma}

\begin{theorem}\label{subfieldinDab}
$\mathbb{F}_{q^{\mu}}$ is the unique largest subfield of $\mathbb{F}_{q^n}$ contained in $D(1,b)$.
\begin{proof}Recall $q=p^l$. Note that $\mathbb{F}_{q^\mu}=\mathbb{F}_{p^{l\mu}}\subseteq D(1,b)$ from Theorem~\ref{vectorSubspacequ}. Suppose $\mathbb{F}_{p^m}$ is a subfield of $\mathbb{F}_{q^n}$, i.e. $m\mid ln$.
We claim the following are equivalent.  
\begin{enumerate}[(1)]
 \item $\mathbb{F}_{p^m}\subseteq D(1,b)$.
\item  $g^{\frac{q^n-1}{p^m-1}}\in D(1,b)$ and $g^{\frac{q^n-1}{p^m-1}(q^\mu-1)}=1$.
\item  $m\mid l\mu$.
\end{enumerate}

Note $g^{\frac{q^n-1}{p^m-1}}$ generates all non-zero elements of $\mathbb{F}_{p^m}$. As  a consequence of Lemma~\ref{Dabsq} We have $(1) \Rightarrow (2)$. Next $(2)$ implies $p^m-1 \mid q^\mu-1=p^{l\mu}-1$ implies $m\mid l\mu$, that is $(3)$. Finally $(3)$ implies $\mathbb{F}_{p^m}\subseteq \mathbb{F}_{p^{l\mu}}\subseteq D(1,b)$, and thus we have $(1)$. 
\end{proof}
\end{theorem}

\begin{theorem}
There is a unique smallest subfield $\mathbb{F}_{q^\omega}$ of $\mathbb{F}_{q^n}$ containing $D(1,b)$ where $\mu \mid \omega\mid n$.
\begin{proof}
Note the intersection of two subfields is a subfield and a subfield containing $D(1,b)$ must contain $\mathbb{F}_{q^{\mu}}$. Thus $\mathbb{F}_{q^\omega}=\cap \{\mathbb{F}_{q^m}\mid  D(1,b)\subseteq \mathbb{F}_{q^m}, \mu \mid m\mid n\}$ is the desired subfield.
\end{proof}
\end{theorem}

We know $\mathbb{F}_{q^\mu} \subseteq D(1,b)\subseteq \mathbb{F}_{q^\omega}$ for some $\mu \mid \omega \mid n$. In particular $\mathbb{F}_{q^1} \subseteq D(1,b)\subseteq \mathbb{F}_{q^n}$.

Denote by $p^\alpha \mid\!\mid k$ if $p^\alpha\mid k$ but $p^{\alpha+1}\nmid k$. In the proof of \cite[Theorem 1]{HULL62}, some properties of $\dfrac{q^k-1}{q-1}\pmod{n}$ were obtained. We  restate the result in the following for convenience.

\begin{theorem}\label{HullDobell1}Suppose $(q,n)$ is a Dickson Pair. The set $\left\{\dfrac{q^k-1}{q-1}\right\}_k$ forms a complete residue system modulo $n$ whenever $k\geq 0$ runs through a complete residue system modulo $n$. Furthermore $\dfrac{q^k-1}{q-1}\equiv 0 \pmod n$ if and only if $k\equiv 0 \pmod n$.
\end{theorem}

\begin{lemma}Let $(q,n)$ be a Dickson Pair.
Suppose $t_b\equiv \dfrac{q^t-1}{q-1} \pmod{n}$ and  $s_b\equiv \dfrac{q^s-1}{q-1} \pmod{n}$. Then $\gcd(s_b,n)=\gcd(s,n)$ and $\gcd(t_b,n)=\gcd(t,n)$. Hence $\gcd(s_b,t_b,n)=\gcd(s,t,n)$. In particular, $\gcd(s_b,t_b,n)=1$ if and only if $\gcd(s,t,n)=1$.
\begin{proof}Let $p^\alpha \mid\!\mid \gcd(s_b,n)$ where $p$ is a prime. From the assumption we have $0\equiv  \dfrac{q^s-1}{q-1} \pmod{p^{\alpha}}$.
Therefore $p^\alpha \mid s$ by Theorem~\ref{HullDobell1} and thus $\gcd(s_b,n)\mid \gcd(s,n)$. By the same token, we have $\gcd(s,n)\mid \gcd(s_b,n)$. As a consequence we have $\gcd(s_b,n)=\gcd(s,n)$. By symmetry we have $\gcd(t_b,n)=\gcd(t,n)$.
The rest follows easily.
\end{proof}
\end{lemma}

The above lemma indicates that
\begin{corollary}  If $b=g^{s_b}$ and $1+b=g^{t_b}$ are non-zero elements of $DN_g(q,n)$ such that $D(1,b)$ is not a subfield of $\mathbb{F}_{q^n}$, then $\mu=\gcd(s_b,t_b,n)$.
\end{corollary}

The following table lists some examples such that $b\in DN_g(q,n)$, $u\in \mathbf{B}\subseteq D(1,b)$ but $u^2 \not \in D(1,b)$. The square closed set $\mathbf{B}_1$ (those inside $[\cdots ]$) of $\mathbf{B}$.  One should note that the bases $\mathbf{B}$ obtained in the table might not yield the largest square closed sets $\mathbf{B}_1$ and thus $\lvert \mathbf{B}_1\rvert\leq l\mu$. Recall $\mu=\gcd(s,t,n)$ and $\mathbb{F}_{q^\mu} \subseteq D(1,b)\subseteq \mathbb{F}_{q^\omega}$. Also $b=g^{s_b}\in g^{\frac{q^s-1}{q-1}}H$ and $1+b=g^{t_b}\in g^{\frac{q^t-1}{q-1}}H$ where  $H=\langle g^n \rangle$ the multiplicative subgroup of ($\mathbb{F}^*_{q^n},\cdot)$ generated by $g^n$. Those entries that have $s_b\not\equiv s \pmod{n}$\footnote{All examples shown in the table have $s_b\equiv s\pmod{n}$.}, $t_b\not\equiv t \pmod{n}$\footnote{$t_b\not\equiv t \pmod{n}$ is not true in general. See the pair $(q,n)=(5,8)$.}, $\gcd(s_b,n)>1$, $\gcd(t_b,n)>1$, $\gcd(s_b,t_b)>1$, $\gcd(s_b,n)\neq \gcd(s,n)$, or $\gcd(t_b,n)\neq \gcd(t,n)$ are included in footnotes.

\pagebreak

\begin{center}
\begin{longtable}[H]{|p{.12\linewidth}| p{.14\linewidth}|p{.31\linewidth}| p{.31\linewidth}|} 

 \hline 
  \rule{0pt}{3ex}$(q,n) \newline
 s_b,s\newline t_b,t,
\newline \mu,\omega$& Primitive poly  \newline $p(x)$ over $\mathbb{F}_p$ & $\mathbf{B}=\cdots,[\mathbf{B}_1]$   &$b=g^{s_b},1+b=g^{t_b}, u, u^2$ \\ 
\hlineB{4}
 \rule{0pt}{3ex}$(4,9) \newline 18352,1\newline 182263, 4\newline 1,9$  & $x^{18}+ x^{12} + x^{10} + x + 1$ &$g^{8135}=g^{17} + g^{13} + g^{10} + g^7 + g^6 + g^4 + g^3 + g, \newline g^{64455}=g^{16} + g^{14} + g^{13} + g^{12} + g^{10} + g^9 + g^6 + g^5 + g^4 + g^3, \newline g^{53440}=g^{15} + g^{14} + g^{13} + g^{11} + g^{10} + g^4 + g^3, \newline [1]$
 & $g^{18352}=g^{13} + g^{10} + g^9 + g^8 + g^4 + g^3 + g,\newline  g^{182263}, \newline
g^{8135}=g^{17} + g^{13} + g^{10} + g^7 + g^6 + g^4 + g^3 + g, \newline g^{16270}=g^{17} + g^{12} + g^{11} + g^{10} + g^8 + g^6 + g^5 + g^4 + g^2$\\  
 \hline
\rule{0pt}{3ex} $(7,6)$\footnote{$\gcd(t_n,n)=\gcd(t,n)=2$}\newline $75073,1\newline 27850,4 \newline 1,6$ & $x^6 + x^4 + 5x^3 + 4x^2 + 6x + 3$ &$g^{55220}=g^5+g^4+5g^3+6g^2, \newline [1]$
 & $g^{75073}=g^3 + 2g^2 + 4g + 5, \newline 
 g^{27850}, \newline g^{55220}=g^5+g^4+5g^3+6g^2, \newline  g^{110440}=2g^5 + 5g^4 + 5g^3 + 3g^2 + 3g + 3$\\ 
 \hline
  \rule{0pt}{3ex} $(5,8)$\footnote{$t_b\not \equiv t \pmod{n},\gcd(s_b,n)=\gcd(s,n)=4, \gcd(s_b,t_b)=21$}\newline$ 258132,4\newline 267435,7 \newline 1,8\newline 267435\equiv 3 \pmod{8}$ & $x^8 + x^4 + 3x^2 + 4x + 2$ &$g^{105697}=g^7 + 4g^6 + 2g^4 + 2g^3 + 4g^2 + g, \newline [1]$
 & $g^{258132}=g^6 + 4g^5 + 3g^4 + g + 2,\newline g^{267435}, \newline g^{105697}=g^7 + 4g^6 + 2g^4 + 2g^3 + 4g^2 + g,\newline g^{211394}=4g^7 + 4g^6 + 2g^5 + g^4 + g^3 + 2g^2 + 2g + 2 $\\
 \hline
\rule{0pt}{3ex}$(8,7)\newline 544237,1\newline 1711762,3\newline 1,7$ & $x^{21} + x^6 + x^5 + x^2 + 1$ &$g^{1896446}=g^{20} + g^{16} + g^{15} + g^{14} + g^{13} + g^{12} + g^8 + g^7 + g^3, \newline g^{1297260}=g^{19} + g^{15} + g^{14} + g^{11} + g^9 + g^8 + g^7 + g^4 + g^3 + g^2 + g, \newline g^{700381}=g^{18} + g^{16} + g^{13} + g^{11} + g^9 + g^7 + g^6 + g^3 + g^2, \newline g^{364372}=g^{17} + g^{16} + g^{15} + g^{13} + g^{12} + g^8 + g^6 + g^4 + g, \newline g^{1136601}=g^{10} + g^6 + g^4 + g^2,\newline [1]$
 & $g^{544237}=g^{17} + g^{14} + g^{11} + g^9 + g^8 + g^7 + g^6 + g^5 + g^3,\newline g^{1711762},\newline g^{1896446}=g^{20} + g^{16} + g^{15} + g^{14} + g^{13} + g^{12} + g^8 + g^7 + g^3,\newline g^{1695741}=g^{19} + g^{17} + g^{15} + g^{12} + g^{11} + g^9 + g^6 + g^4 + g^2 + 1 $\\ 
 \hline
 \rule{0pt}{3ex} $(13,6)$\footnote{$\gcd(s_b,n)=\gcd(s,n)=2$}\newline $134128,4\newline 327229, 1\newline 1,6$  & $x^6 + 10x^3 + 11x^2 + 11x + 2$ &$g^{3773368}=g^5 + 3g^4 + 11g^3 + 3g^2, \newline [1]$
 & $g^{134128}=2g^2 + 2,\newline g^{327229} ,\newline g^{3773368}=g^5 + 3g^4 + 11g^3 + 3g^2,\newline g^{2719928}=4g^5 + 7g^4 + g^3 + 3g^2 + 12g + 7 $\\ 
 \hline
\rule{0pt}{3ex} $(7,9)$\footnote{$\gcd(s_b,n)=\gcd(s,n)=3$}\newline $11634549,6\newline 7471843,7,\newline 1,9$  & $ x^9 + 6x^4 + x^3 + 6x + 4
$ &$g^{32509044}=g^8 + 5g^7 + g^6 + 2g^5 + 5g^4 + 4g^3 + 5g^2 + 3g, \newline [1]$
 & $g^{11634549}=3g^5 + 3g^4 + 6g^3 + 6g^2 + 3g + 4,\newline g^{7471843},\newline g^{32509044}=g^8 + 5g^7 + g^6 + 2g^5 + 5g^4 + 4g^3 + 5g^2 + 3g,\newline g^{24664482}=6g^8 + 5g^7 + 6g^6 + 3g^4 + g^3 + 5g^2 + 5g + 2 $\\
 \hline

\rule{0pt}{3ex} $(19,6)$\footnote{$\gcd(t_b,n)=\gcd(t,n)=2$}\newline $28108087,1\newline 20766196,4, \newline 1,6$  & $ x^6 + 17x^3 + 17x^2 + 6x + 2
$ &$g^{12162079}=g^5 + 2g^4 + 8g^3 + g^2 + 15g, \newline [1]$
 & $g^{28108087}=7g^3 + 6g^2 + 14g + 15, \newline g^{20766196},\newline g^{12162079}=g^5 + 2g^4 + 8g^3 + g^2 + 15g,\newline  g^{24324158}=4g^5 + 4g^4 + 14g^2 + 2g + 12 $\\
 \hline
\rule{0pt}{3ex} $(13,8)\footnote{$\gcd(s_b,t_b)=3$}\newline 265720629,5\newline 222796497,1, \newline 1,8$  & $ x^8 + 8x^4+12x^3 +2x^2 + 3x + 2
$ &$g^{646885587}=g^7 +6g^6+6g^5 + 11g^4 +11g^3 + 8g^2 + 12g, \newline [1]$
 & $g^{265720629}=8g^6+8g^5+12g^4+g^3+3g^2+3g+6 ,\newline g^{222796497} ,\newline   g^{646885587}= g^7 +6g^6+6g^5 + 11g^4 + 11g^3 + 8g^2 + 12g,\newline  g^{478040454}=11g^7 +11g^6+8g^5 + 10g^3 + 7g^2 + 10g + 2$\\
 \hline 
\caption{Examples when $D(1,b)$ is not a field}
\label{tablenotfield}
\end{longtable}
\end{center}

\begin{theorem}\label{primitivenormalbasisfield}\cite{LenstraSchoof87}
For every prime power $q > 1$ and every positive integer $m$ there exists a
primitive normal basis of $\mathbb{F}_{q^m}$  over $\mathbb{F}_{q}$.
\end{theorem}

Inspired by Lenstra and Schoof~\cite{LenstraSchoof87} we introduce following notation. 

\begin{definition}Suppose $\mathbf{B}$ is a basis of the space $D(1,b)$ 
 over $\mathbb{F}_p$ such that $\mathbf{B}$ is of the form  $\{u, u^p, \ldots, u^{p^{m-1}}\}$, that is a basis consisting of all the algebraic conjugates of a  primitive element $u$ of $\mathbb{F}_{q^n}$, with $l\mu \leq m$. We  call the basis $\mathbf{B}$ {\em primitive normal} and the element $u$ a {\em primitive normal element} of $D(1,b)$ over $\mathbb{F}_p$.
\end{definition}

\begin{question}  Are the following equivalent?
  \begin{enumerate}
    \item There is a primitive  normal basis of $D(1,b)$ over $\mathbb{F}_p$.
    \item $D(1,b)=\mathbb{F}_{p^{l\mu}}$.
      \end{enumerate}
\end{question}
It is clear that $2. \implies 1.$ is direct consequence of Theorem~\ref{primitivenormalbasisfield}.

\bibliographystyle{plain}

\end{document}